\documentclass[a4paper,12pt]{amsart}
\usepackage{amssymb}
\usepackage{ifthen}
\usepackage{graphicx}
\usepackage{float}
\nonstopmode
\newtheorem{thm}{Theorem}
\newtheorem{cor}{Corollary}
\newtheorem{lem}{Lemma}

\newtheorem{conj}{Conjecture}
\newtheorem{prob}{Problem}

\theoremstyle{definition}
\newtheorem{defn}{Definition}[section]
\newtheorem{example}{Example}
\newtheorem{remark}{Remark}

\newenvironment{pf}[1][]{%
 \vskip 1mm
 \noindent
 \ifthenelse{\equal{#1}{}}%
  {{\slshape Proof. }}%
  {{\slshape #1.} }%
 }%
{\qed\medskip}

\newcounter{alphabet}
\newcounter{tmp}
\newenvironment{Thm}[1][]{\refstepcounter{alphabet}%
\bigskip%
\noindent%
{\bf Theorem \Alph{alphabet}}%
\ifthenelse{\equal{#1}{}}{}{ (#1)}%
{\bf .} \itshape}{\vskip 8pt}

\makeatletter
\newcommand{\Ref}[1]{\@ifundefined{r@#1}{}{\setcounter{tmp}{\ref{#1}}\Alph{tmp}}}
\makeatother

\newcommand{\IR}{{\mathbb R}}

\newcommand{\IC}{{\mathbb C}}
\newcommand{\ID}{{\mathbb D}}

\newcommand{\real}{{\operatorname{Re}\,}}

\def\be{\begin{equation}}
\def\ee{\end{equation}}

\newcommand{\bee}{\begin{enumerate}}
\newcommand{\eee}{\end{enumerate}}

\newcommand{\blem}{\begin{lem}}
\newcommand{\elem}{\end{lem}}
\newcommand{\bthm}{\begin{thm}}
\newcommand{\ethm}{\end{thm}}
\newcommand{\bcor}{\begin{cor}}
\newcommand{\ecor}{\end{cor}}
\newcommand{\beg}{\begin{example}}
\newcommand{\eeg}{\end{example}}
\newcommand{\begs}{\begin{examples}}
\newcommand{\eegs}{\end{examples}}
\newcommand{\bdefe}{\begin{defn}}
\newcommand{\edefe}{\end{defn}}
\newcommand{\bprob}{\begin{prob}}
\newcommand{\eprob}{\end{prob}}
\newcommand{\bques}{\begin{ques}}
\newcommand{\eques}{\end{ques}}
\newcommand{\bei}{\begin{itemize}}
\newcommand{\eei}{\end{itemize}}
\newcommand{\bcon}{\begin{conj}}
\newcommand{\econ}{\end{conj}}
\newcommand{\bcons}{\begin{conjs}}
\newcommand{\econs}{\end{conjs}}
\newcommand{\bprop}{\begin{propo}}
\newcommand{\eprop}{\end{propo}}
\newcommand{\br}{\begin{remark}}
\newcommand{\er}{\end{remark}}
\newcommand{\brs}{\begin{rems}}
\newcommand{\ers}{\end{rems}}
\newcommand{\bo}{\begin{obser}}
\newcommand{\eo}{\end{obser}}
\newcommand{\bos}{\begin{obsers}}
\newcommand{\eos}{\end{obsers}}
\newcommand{\bpf}{\begin{pf}}
\newcommand{\epf}{\end{pf}}
\newcommand{\ba}{\begin{array}}
\newcommand{\ea}{\end{array}}
\newcommand{\beq}{\begin{eqnarray}}
\newcommand{\beqq}{\begin{eqnarray*}}
\newcommand{\eeq}{\end{eqnarray}}
\newcommand{\eeqq}{\end{eqnarray*}}

\begin{document}
\bibliographystyle{amsplain}
\title[Univalence criterion for harmonic mappings and $\Phi$-like functions]{Univalence criterion for harmonic mappings and $\Phi$-like functions}

\author{Sergey Yu. Graf, Saminathan Ponnusamy, and Victor V. Starkov}

\address{S. Yu. Graf,
Department of Mathematics, Tver State University, ul. Zhelyabova 33, Tver, 170000, Russia.
}
\email{sergey.graf@tversu.ru}

\address{S. Ponnusamy,
Indian Statistical Institute (ISI), Chennai Centre, SETS (Society
for Electronic Transactions and Security), MGR Knowledge City, CIT
Campus, Taramani, Chennai 600 113, India.
}
\email{samy@isichennai.res.in, samy@iitm.ac.in}

\address{V. V. Starkov,
Department of Mathematics, University of Petrozavodsk,
ul. Lenina 33, Petrozavodsk, 185910 , Russia
}
\email{vstarv@list.ru }


\begin{abstract}
In this paper, we obtain a new characterization for univalent harmonic mappings and obtain a structural formula for the associated function
which defines the analytic $\Phi$-like functions in the unit disk. The new criterion stated in this article for the injectivity of harmonic mappings
implies the well-known results of Kas'yanyuk \cite{Kas59} and Brickman \cite{Brick73} for analytic functions, but with a simpler
proof than theirs.  A number of consequences of the characterization, and examples are also presented.
Further investigation provides a new method to construct univalent harmonic mappings with the help of an improved distortion theorem.
\end{abstract}

\thanks{}

\maketitle
\pagestyle{myheadings}
\markboth{S. Yu. Graf, S. Ponnusamy, and V. V. Starkov}{Univalence criterion for harmonic mappings}

\section{Introduction}
The article is devoted to the investigation of complex-valued harmonic functions defined on a simply connected domain $D$
of the complex plane $\IC$. Here we say that $f$ is harmonic in $D$ if the real and imaginary parts of $f$
satisfy the Laplace  equation.
Evidently, $f$ is harmonic on $D$ if and only if it has a decomposition $f=h+\overline{g}$, where $h$ and $g$ are analytic on $D$. Here
$h$ and $g$ are called analytic and co-analytic parts of $f$, respectively.
In function theoretic point of view, the study of univalent harmonic mappings along with geometric subfamilies was pioneered in 1984 by
Clunie and Sheil-Small \cite{Clunie-Small-84}. In recent years, this topic has received the attention of many and the literature is now vast
(see the monograph \cite{Duren:Harmonic}, and also the recent expository article by Ponnusamy and Rasila \cite{PonRasi2013}).
Recently, some of the results from conformal case has been generalized to the case of planar harmonic mappings and also
to the case of functions of several variables. However, some others have no counterparts and thus have many challenging problems and
conjectures remain unsolved (see for example, \cite{BL}). Nevertheless, the analogy to the theory of conformal mappings is far from obvious
and the family of univalent harmonic mappings is much larger than its analytic counterparts.

The main goal in this article is to obtain criterion for the univalence of harmonic mappings and related results. In the particular case, this criterion leads to the known concept of the so-called $\Phi$-like analytic functions.

Let $\mathcal{A}(a)$ denote the set of functions $f$ analytic in the unit disk $\ID=\{z \in \IC:\, |z|<1\}$ with the form
$$f(z)=a+\sum_{n=1}^\infty a_nz^n,\quad f'(0)\neq0.
$$
In particular, let $\mathcal{A}=\{f\in \mathcal{A}(0):\, f'(0)=1\}$ and $\mathcal{S}=\{f\in \mathcal{A}:\, ~\mbox{$f$ is univalent in $\ID$}\}$.
Denote by  $\mathcal{S}^\star$, the subfamily of functions $f\in \mathcal{S}$ such that $f(\ID)$ are starlike with respect to  the origin.
Recall that function $f\in \mathcal{A}$ is called $\Phi$-like (in $\ID$) if and only if
\be\label{GPS2-eq2d}
{\rm Re}\left (\frac{zf'(z)}{\Phi(f(z))}\right )>0,\quad z\in\ID,
\ee
where $\Phi$ is analytic on $f(\ID)$, $\Phi(0)=0$, and ${\rm Re}\,\Phi'(0)>0.$
The concept of $\Phi$-like functions was introduced by Kas'yanyuk \cite{Kas59} in 1959 and
independently in 1973  by Brickman \cite{Brick73}.
The reader is referred to \cite {AvkAks75,Hotta13,Rus76} for related investigations about $\Phi$-like functions.
Surprisingly,  every $\Phi$-like function is univalent in $\ID$ and, every $f\in \mathcal{S}$ is $\Phi$-like for some $\Phi$.
 Evidently, $f\in \mathcal{S}^\star$ is a special case of $\Phi$-like function with $\Phi(w)=w$;
and $f\in \mathcal{S}_p$ is a special case with $\Phi(w)=e^{i\alpha}w$ and $\alpha\in (-\pi/2,\pi/2)$.
Here $\mathcal{S}_p$ represents the family of all spiral-like functions $f\in \mathcal{A}$;
i.e.,  for each $f\in\mathcal{S}_p$ there exists an $\alpha\in (-\pi/2,\pi/2)$ such that
$${\rm Re}\left (e^{-i\alpha}\frac{zf'(z)}{f(z)}\right )>0,\quad z\in\ID.
$$


The paper is organized as follows. A complete characterization of univalent harmonic mappings is presented in Section  \ref{GPS2-sec2} and
we use this condition to obtain a number of consequences of it. In Theorem \ref{GPS2-th2}, we establish a structural formula for analytic univalent $\Phi$-like mappings of the unit disk. In Section \ref{GPS2-sec3}, we obtain an improved distortion theorem (Lemma \ref{GPS2-th2n1}) and present a method of construction of univalent harmonic mappings (Theorem \ref{GPS2-th3}).

\section{Main results and proofs}\label{GPS2-sec2}
Now, we recall the two recent results which provide sufficient conditions for a harmonic function to be close-to-convex (univalent)
in $\ID$. A harmonic function $f$ defined on $\ID$ is called convex (resp. {\it close-to-convex}) if it is univalent in $\ID$ and  $f(\ID)$ is a convex (resp.  {\it close-to-convex}) domain. Recall that a domain $D\subset\mathbb{C}$ is called
close-to-convex if its complement $\mathbb C\setminus D$ can be written as an union of rays that can intersect only at their end points.
We say that a harmonic function $f=h+\overline{g}$ on $\ID$ is normalized, denoted by $f\in\mathcal{H}$, if  $h(0) = g(0) = 0$ and $h'(0) = 1$.
A harmonic function $f=h+\overline{g}$ on $\ID$ is called sense-preserving if the Jacobian $J_f(z)=|h'(z)|^2 -|g'(z)|^2$ is positive in $\ID$.


\begin{Thm}
\label{uni-theo1}
{\rm \cite{Hiroshi-Samy-2010}}
Suppose $f=h+\overline{g}$ is a harmonic mapping in a convex domain $\Omega$ such that
${\rm Re\,} (e^{i\gamma}h'(z))>|g'(z)|$ for all $z\in \Omega$, and for some $\gamma \in \mathbb{R}$.
Then $f$ is close-to-convex and univalent in $\Omega$.
\end{Thm}

Later this result has been generalized  in \cite{PonSai-1(11)} as follows.

\begin{Thm}\label{APS1-11-lem1}
Let $f=h+\overline{g}\in\mathcal{H}$. 
Further, let $G$ be univalent, analytic and convex in $\ID$. If $f$ satisfies
\be\label{APS1-11-eq1a}
{\rm Re}\left (e^{i\gamma}\frac{h'(z)} {G'(z)}\right )>\left|\frac{g'(z)}{G'(z)}\right| ~\mbox{ for all $z\in \ID$ and
for some $\gamma$ real},
\ee
then $f$ is sense-preserving univalent and close-to-convex in $\ID$.
\end{Thm}

However, univalency of the harmonic mappings $f=h+\overline{g}$ in Theorem \Ref{APS1-11-lem1} was shown by Mocanu \cite{Mocanu80}.
We refer to \cite{PonSai-1(11),Hiroshi-Samy-2010} for a proof and applications of Theorems \Ref{uni-theo1} and \Ref{APS1-11-lem1}.
By analogy with the known criterion of I.E.~Basilevich for the univalence of analytic functions,
a criterion for the univalence of harmonic functions $f$ in terms of the series of the analytic and the co-analytic parts of $f$
was obtained in \cite{star2014}.

Now, we state  one of our main results -- another criterion for injectivity -- harmonic analog of $\Phi$-like mappings,  and some of its consequences.

\begin{thm}\label{GPS2-th1}
Let $f=h+\overline{g}$ be harmonic on a convex domain $D$ and $\Omega=f(D)$. Then $f$ is univalent in $D$
if and only if there exists a complex-valued function $\phi=\phi(w,\overline{w})$ in $C^1(\Omega)$ and such that for every
$\epsilon  \in \partial \ID$ there exists a real number $\gamma=\gamma (\epsilon )$ satisfying
\be\label{GPS2-eq1}
{\rm Re}\,\big \{e^{i\gamma}\big(\partial\phi(f(z),\overline{f(z)})+\epsilon \overline{\partial}\phi(f(z),\overline{f(z)})\big)\big \}>0 ~\mbox{ for all $z\in D$},
\ee
where $\partial=\frac{\partial}{\partial z}$ and $\overline{\partial}=\frac{\partial}{\partial \overline{z}}$.
\end{thm}
\begin{proof}
To prove the univalency of $f$ under the sufficient condition \eqref{GPS2-eq1}, we shall apply the simple idea that was used by
Noshiro and Warschawski for constructing their well-known sufficient condition of univalence of analytic functions
(cf. \cite[Chapter 2, Theorem 2.16]{Duren:Analytic}).

Let $z_1$ and $z_2$ be two distinct points in $D$.
Then, because $D$ is a convex domain, the line segment $[z_1,z_2]$, parameterized by $z(t) =(1-t)z_1 + tz_2\in D$ for $0\leq t\leq 1$,
lies in $D$, and $z'(t)=z_2-z_1$. Define $\Psi (z)=\phi(f(z),\overline{f(z)})$. It follows that
\beqq
\Psi  (z_2) - \Psi (z_1) &=& \int_{0}^{1} \frac{d}{dt}\Psi (z(t))\,dt \\
&= &(z_2 - z_1)\int_{0}^{1} \left[\Psi _z (z(t)) + \frac{\overline{z_2 - z_1}}{z_2 - z_1}\Psi _{\overline{z}}(z(t))\right ] dt.
\eeqq
Set $\epsilon =(\overline{z_2 - z_1})/(z_2 - z_1)$ and observe that $|\epsilon|=1$. From \eqref{GPS2-eq1} there exists a $\gamma$ depending on $\epsilon$ such that
\begin{eqnarray*}
{\rm Re\,} \left (e^{i\gamma} \frac{\Psi  (z_2) - \Psi (z_1)}{z_2-z_1} \right ) & = &
\int_{0}^{1}  {\rm Re\,} \left (e^{i\gamma}\big(\partial\phi(f(z),\overline{f(z)})+\epsilon \overline{\partial}\phi(f(z),\overline{f(z)})\big)\right )dt> 0
\end{eqnarray*}
which proves the univalency of $\Psi$, i.e.  $\phi$ is univalent on $\Omega$  implying the univalence of $f$ on $D$.

Conversely, suppose that $f$ is univalent in $D$ and  $\Omega=f(D)$. Denote $f^{-1}(w)=\phi(w,\overline{w})$ for $w\in \Omega$. Then $f^{-1}\in C^1(\Omega)$ and
$\phi(f(z),\overline{f(z)})=z$ for $z\in D$. Observe that $\partial \phi(f(z),\overline{f(z)})=1$ and $\overline{\partial}\phi(f(z),\overline{f(z)})=0$,
showing that \eqref{GPS2-eq1} holds for $\gamma =0$, for example. The proof is complete.
\end{proof}

We note that this criterion of injectivity of harmonic mappings implies the known result of Kas'yanyuk \cite{Kas59} and Brickman \cite{Brick73} for analytic case, but the proof stated above is more natural and essentially shorter.

By taking into account proof of the last theorem, it is possible to reformulate the criterion of injectivity in more simple form.

\bcor\label{GPS2-cor1}
Let $f=h+\overline{g}$ be harmonic on a convex domain $D$ and $\Omega=f(D)$.  Function $f$ is univalent in $D$ if and only if there exists a complex-valued function $\phi\in C^1(\Omega)$ such that
\be\label{GPS2-eq2}
{\rm Re}\,\partial\phi(f(z),\overline{f(z)})> \big |\overline{\partial}\phi(f(z),\overline{f(z)})\big |  ~\mbox{ for all $z\in D$}.
\ee
\ecor

We have to note, that criteria from Theorem \ref{GPS2-th1} and Corollary \ref{GPS2-cor1} are not equal since as the sets of functions  $\phi$ used in conditions 
\eqref{GPS2-eq1} and  \eqref{GPS2-eq2} do not coincide.

We observe that univalence part of Theorem \Ref{uni-theo1} follows from Corollary \ref{GPS2-cor1}, and thus, it is natural to determine a condition of the type \eqref{GPS2-eq2} so that the corresponding $f$ is close-to-convex.


In the analytic case (i.e. $f(z)=h(z)$ and $g(z)\equiv 0$, $z\in\ID$), from Theorem \ref{GPS2-th1}, we conclude the following:

\begin{cor}\label{GPS2-cor1b}
Let $f$ be analytic in a convex domain $D$ and $\Omega =f(D)$. Then $f$ is univalent in $D$
if and only if there exists an analytic function  $\phi(w)$ on $\Omega$ such that
\be\label{GPS2-eq3}
{\rm Re}\,\big \{\frac{d}{dz}\big (\phi(f(z) \big ) \big \}={\rm Re}\,\big \{\phi'(f(z))f'(z)  \big \}> 0 ~\mbox{ for all $z\in D$.}
\ee
\end{cor}

Corollary \ref{GPS2-cor1b} reduces to the result of  Kas'yanyuk \cite{Kas59} and Brickman \cite{Brick73}. Indeed if, 
in Corollary \ref{GPS2-cor1b}, we introduce $\Phi$ by
$$\Phi (w)=\frac{f^{-1}(w)}{e^{i\gamma}\phi'(w)}, ~w\in \Omega =f(D),
$$
then \eqref{GPS2-eq3} may be rewritten in an equivalent form:
$${\rm Re}\left (\frac{zf'(z)}{\Phi(f(z))}\right )>0,\quad z\in D,
$$
which in the case of $f\in {\mathcal A}$ leads to $\Phi$-like function defined by \eqref{GPS2-eq2d}. That is, we have

\bcor
If $f\in {\mathcal A} $, then $f$ is
univalent in $\ID$ if and only if $f$ is $\Phi$-like for some $\Phi$.
\ecor


The natural  question, that can be asked, is how large the set of functions $\phi$ satisfying condition \eqref{GPS2-eq1} in
Theorem \ref{GPS2-th1} for a given function $f$.
The following theorem gives us a complete characterization of such functions $\phi (w)$ in the analytic case.

\bthm[Structural formula]\label{GPS2-th2}
Let $f$ be analytic and univalent in $\ID$. Then the inequality \eqref{GPS2-eq3} holds for the analytic function
$\phi (w)$ in $f(\ID)$ if and only if
\be\label{GPS2-eq4}
\phi (w)= -c\left[(1+ic_1)f^{-1}(w)+2\int_0^{2\pi}e^{-i\theta}\log(1-f^{-1}(w)e^{i\theta})\,d\mu(\theta)\right]+c_0,
\ee
where $c>0,\,c_1\in\IR,\,c_0\in\IC$ are arbitrary constants, $\mu(\theta)$ is an arbitrary real-valued increasing (in the wide sense) on $[0,2\pi]$ function  of total variation equal to $1$.
\ethm
\begin{proof}
Suppose that the condition \eqref{GPS2-eq3} holds for a given analytic function $f$ on $\ID$ and for some analytic function $\phi$ on $f(\ID)$.
Then $\frac{d}{dz}\phi(f(z))$ can be represented as a power series of the form
$$\frac{d}{dz}\phi(f(z))=a_0+a_1z+ \cdots 
$$
in $\ID$. If we let $a_0=c+ic'$ with $c=\real a_0$, then $c>0$ by the assumption \eqref{GPS2-eq3}. Define $p$ by
$$ p(z):=\frac{1}{c} \frac{d}{dz}\phi(f(z))-i\frac{c'}{c},
$$
Then $p$ is analytic in $\ID$, $p(0)=1$ and $\real p(z)>0$ in $\ID$. Applying the Herglotz representation we obtain that
\be\label{GPS2-eq2a}
p(z)=\int_0^{2\pi}\frac{1+ze^{i\theta}}{1-ze^{i\theta}}\,d\mu(\theta),
\ee
where $\mu(\theta)$ is real-valued function satisfying the condition of Theorem \ref{GPS2-th2}. Then
\be\label{GPS2-eq2b}
\phi(f(z))=c\int_0^zp(s)\,ds+ic'z+c_0=cz\int_0^1p(tz)\,dt+ic'z+c_0,
\ee
where $c_0\in\IC$ is an arbitrary constant. 
Putting  $z=f^{-1}(w)$ in 
\eqref{GPS2-eq2b} for $w\in f(\ID)$, we obtain
\begin{align*}
\phi(w)&=f^{-1}(w)\left[c\int_0^1p(tf^{-1}(w))\,dt+ic'\right]+c_0\\
&=cf^{-1}(w)\left[\int_0^{2\pi}\left(\int_0^1\frac{1+tf^{-1}(w)e^{i\theta}}{1-tf^{-1}(w)e^{i\theta}}
\,dt\right )d\mu(\theta)+i\frac{c'}{c}\right]+c_0\\
&=cf^{-1}(w)\left[\int_0^{2\pi}\left(-1-2\frac{\log(1-f^{-1}(w)e^{i\theta})}{f^{-1}(w)e^{i\theta}}\right )d\mu(\theta)-ic_1\right]+c_0,
\end{align*}
where $c_1$ is real constant. Therefore
$$\phi (w)= -c\left[(1+ic_1)f^{-1}(w)+2\int_0^{2\pi}e^{-i\theta}\log(1-f^{-1}(w)e^{i\theta})\,d\mu(\theta) \right]+c_0.
$$

Conversely, if $f$ is univalent in $\ID$ and function $\phi$ is given by \eqref{GPS2-eq4} for some arbitrary $c>0$, $c_1\in\IR$, $c_0\in\IC$ 
and the function $\mu(\theta)$ satisfying conditions of Theorem \ref{GPS2-th2}, then it is clear that
$$ \frac{d}{dz}\phi(f(z))=-c\left[ 1+ic_1-2\int_{0}^{2\pi}\frac{d\mu(\theta)}{1-ze^{i\theta}}\right]=cp(z)-ic\cdot c_1,
$$
where the function $p(z)$ has the Herglotz representation of the form \eqref{GPS2-eq2a}. Therefore, the condition \eqref{GPS2-eq3} is 
satisfied for $\phi$ and we complete the proof.

\end{proof}


%
%

\section{Construction of univalent harmonic mappings}\label{GPS2-sec3}
Let $\mathcal{S}_H$ denote the family of functions in $\mathcal{H}$ that are univalent in $\ID$, and  $\mathcal{S}_H^0$ denote
the subfamily of functions $f\in\mathcal{S}_H$ with the additional normalization $f_{\overline{z}}(0)=0$.
Thus, each $f=h+\overline{g}\in {\mathcal S}_{H}^{0}$ has the expansion
$$f(z)=z+\sum\limits_{n=2}^{\infty }a_{n}z^{n}+\overline{\sum\limits_{n=2}^{\infty }b_{n}z^{n}}, \quad z\in \ID.
$$
The families ${\mathcal S}_H$ and ${\mathcal S}_{H}^{0}$ are known to be normal with respect to the topology of uniform convergence on
compact subsets of $\ID$, whereas only ${\mathcal S}_{H}^{0}$ is compact.


In order to construct univalent harmonic mappings, we need the following lemma which is indeed an improved version
of a similar result  presented by Starkov \cite{star2014}. We refer to \cite{Duren:Harmonic,Sheil90}, for a detailed description of how the order of the
family $\mathcal{S}_H$ determines the bounds on both the maximum and minimum modulus for functions lying in an affine and linear invariant subfamilies of
$\mathcal{S}_H$.

\blem\label{GPS2-th2n1}
Let $f=h+\overline g\in {\mathcal S}_H^0$. Then for all $z_1,z_2\in\{z:\,|z|= r\}$, we have
$$|f(z_2)-f(z_1)|\ge |z_2-z_1|C(r),
$$
where
$$C(r)=\frac{1}{4\alpha r}\left(\frac{1-r}{1+r}\right)^\alpha\left[1-\left(\frac{1-r}{1+r}\right)^{2\alpha}\right]
$$
which is strictly decreasing on $(0,1)$, and $\alpha={\rm ord}\,{\mathcal S}_H=\sup_{f\in {\mathcal S}_H}|h''(0)/2|.$
\elem
\bpf
Let $f=h+\overline g\in {\mathcal S}_H^0$ and $|z_1|=|z_2|= r<1$. There is nothing to prove if
$z_1=z_2$ and so, we may let $z_1\ne z_2$. Consider a conformal automorphism $\phi (z)$ of the unit disk $\ID$ defined by the
formula
$$\phi (z) =\frac{z+z_1}{1+\overline{z_1}z}, \ |z_1|<1,
$$
so that  $\phi (0)=z_1$ and $\phi(z_0)=z_2$  for $z_0=(z_2-z_1)/(1-\overline{z_1}z_2).$
Then the normalized function $F$ defined by
$$F(z)=\frac{(f\circ\phi) (z) -f(z_1)}{h'(z_1)(1-|z_1|^2)}
$$
belongs to ${\mathcal S}_H $ and $F_{\overline z}(0)=\overline{g'(z_1)}/h'(z_1)$. Because $f=h+\overline g$ is sense-preserving with $g'(0)=0$, from the classical Schwarz lemma applied to the dilatation $\omega (z)=g'(z)/h'(z)$, $\omega :\,\ID\rightarrow \ID$,  it follows
that $|\omega (z)|=|g'(z)/h'(z)|\leq |z|$ for $z\in \ID$ and thus, $|F_{\overline z}(0)|\leq |z_1|$. Consequently,
the transformation $A$, defined by the composition of $F$ with an affine mapping,
$$A(z)=\frac{F(z)-F_{\overline z}(0)\overline{F(z)}}{1-|F_{\overline z}(0)|^2}
$$
belongs to ${\mathcal S}_H^0$. By direct calculations we conclude that
\beqq |A(z_0)|&\le & \frac{|F(z_0)|}{1-|F_{\overline z}(0)|}\\
&=&\frac{|f(z_2)-f(z_1)|}{|h'(z_1)|(1-|z_1|^2)(1-|F_{\overline z}(0)|)}\\
&\le &\frac{|f(z_2)-f(z_1)|}{(|h'(z_1)|-|g'(z_1)|)(1-|z_1|^2)}.
\eeqq
It is known by the work of Sheil-Small \cite{Sheil90} that for  $f\in {\mathcal S}_H$,
$$|h'(z)|-|g'(z)|\ge\frac{(1-|z|)^{\alpha-1}}{(1+|z|)^{\alpha+1}},
$$
where $\alpha={\rm ord}\,{\mathcal S}_H.$ Since $|z_1|= r$, using the last two inequalities, one obtains that
\be\label{GPS2-eq2n1}
|f(z_2)-f(z_1)|\ge |A(z_0)|\frac{(1-r)^{\alpha}}{(1+r)^{\alpha}}.
\ee

The lower estimation for $|A(z_0)|$ follows from the growth theorem for ${\mathcal S}_H^0$ due to Sheil-Small \cite{Sheil90}:
$$ |A(z_0)|\ge\frac{1}{2\alpha}\left[1-\left(\frac{1-|z_0|}{1+|z_0|}\right)^{\alpha}\right]=
\frac{1}{2\alpha}\left[1-\left(\frac{1-|(z_2-z_1)/(1-\overline{z_1}z_2)|}{1+|(z_2-z_1)/(1-\overline{z_1}z_2)|}\right)^{\alpha}\right].
$$
It remains to note that in \cite{star2014} the inequality
$$1-\left(\frac{1-|(z_2-z_1)/(1-\overline{z_1}z_2)|}{1+|(z_2-z_1)/(1-\overline{z_1}z_2)|}\right)^{\alpha}\ge
\frac{1}{2r}|z_2-z_1|\left[1-\left(\frac{1-r}{1+r}\right)^{2\alpha}\right]
$$
was proved for $|z_1|=|z_2|=r$. Applying the last two inequalities in \eqref{GPS2-eq2n1} completes the proof of the desired estimation.

To prove the monotonicity of $C(r)$ note that
$$C\left(\frac{1-x}{1+x}\right)=\frac{1}{4\alpha}\frac{1+x}{1-x}x^\alpha(1+x^\alpha)(1-x^\alpha),
$$
where $(1-x)/(1+x)$ is strictly decreasing on $(0,1)$. So the monotonicity of $C(r)$ will be clear if we shall demonstrate that the function 
$$\psi(x)=\frac{1-x^\alpha}{1-x}
$$ 
is increasing on $(0,1)$. Direct calculation shows that
$$\psi'(x)=\frac{1-\alpha x^{\alpha-1}+(\alpha-1)x^\alpha}{(1-x)^2}>0
$$
for all $x\in(0,1)$ and for each $\alpha\ge 1$ in view of monotonicity of numerator on $(0,1)$. Note that $\alpha={\rm ord}\,{\mathcal S}_H\ge 3$ (cf. \cite{Duren:Harmonic}).
So $\psi(x)$ is increasing and hence, $C(r)$ is decreasing on $(0,1)$. This completes the proof of the lemma.
\epf

\br\label{rem1}
If follows from the proof of Lemma \ref{GPS2-th2n1} that if the function $f$ belongs to $\mathcal{S}$, then the order $\alpha$ in the definition of $C(r)$ is equal to $2$, since  ${\rm ord}\,{\mathcal S}=2$ (cf. \cite{Bib} and \cite[Chapter 2, \S 4]{Golu66}). We note that in the
analytic case, the estimation in Lemma \ref{GPS2-th2n1} is not the best known.
\er

\bcor\label{GPS2-cor2n1}
Let $f=h+\overline g$ be a univalent sense-preserving harmonic mapping of the unit disk $\ID$. Then, for all $z_1,z_2\in\{z\in \ID:\,|z|= r\}$,
$$|f(z_2)-f(z_1)|\ge |z_2-z_1|(|h'(0)|-|g'(0)|)C(r),
$$
where $C(r)$ is as in Lemma \ref{GPS2-th2n1}.
\ecor
\bpf
The desired conclusion follows by considering appropriate normalization for $f=h+\overline g$. Indeed if we consider
$$f_1(z)=\frac{f(z)-f(0)}{h'(0)}  ~\mbox{ and }~ f_2(z)=\frac{f_1(z)-(\overline{g'(0)}/h'(0) )\, \overline{f_1(z)}}{1-|g'(0)/h'(0)|^2}
$$
then, by the assumptions, we have $f_1\in {\mathcal S}_H$ and $f_2\in {\mathcal S}_H^0$.
Thus,  for $|z_1|=|z_2|=r<1$, Lemma \ref{GPS2-th2n1} shows that
$$\frac{|f(z_2)-f(z_1)|}{|h'(0)|-|g'(0)|}=\frac{|f_1(z_2)-f_1(z_1)|}{1-|g'(0)/h'(0)|}\ge |f_2(z_2)-f_2(z_1)|\ge
|z_2-z_1|C(r)
$$
from which the desired conclusion follows.
\epf

\br\label{rem2}
Computing the sharp lower estimation of $|f(z_2)-f(z_1)|$ for univalent harmonic mappings $f$ remains an open question.
\er


\bthm\label{GPS2-th3}
Let $f=h+\overline g$ be a univalent sense-preserving harmonic mapping of the unit disk $\ID$.
Let $m(r)=\displaystyle \min_{|z|\le r}(|h'(z)|-|g'(z)|)$ for $r\in [0,1)$.
Let $\varphi=p+\overline q$ be harmonic in $\ID$ such that $\displaystyle A= \sup_{z\in\ID}\big (|p'(z)|+ |q'(z)|\big)<\infty$.
Then the function $F$ defined by
$$F(z)=f(rz)+\varepsilon \varphi(z)
$$
is univalent and sense-preserving in $\ID$ for all $r\in (0,1)$ and for all $\varepsilon$ with
$$0\leq \varepsilon<\frac{r}{A}\min\left\{m(r),m(0)C(r) \right\},
$$
where $C(r)$ is defined as in Lemma \ref{GPS2-th2n1}.
\ethm \bpf
Let $r\in(0,1)$ be fixed. First we prove that $F$ is locally univalent in $\ID$. We begin to observe that
\beqq
|F_z(z)|-|F_{\overline z}(z)|&=&|r h'(rz)+\varepsilon p'(z)|-|r g'(rz)+\varepsilon q'(z)|\\
&\ge& r(|h'(rz)|-|g'(rz)|)-\varepsilon(|p'(z)|+|q'(z)|)\\
&\ge & r m(r)-\varepsilon A>0
\eeqq
for all $\varepsilon<r m(r)/A$. Next, we fix $\rho\in (0,1)$ and show that the function $F$ maps the circle $\gamma_\rho =\{z:\,|z|=\rho\}$
univalently onto a simple closed curve.

In view of Corollary \ref{GPS2-cor2n1} and the univalence of $f$, it follows that
$$
|f(r z_2)-f(r z_1)|\ge r|z_2-z_1|(|h'(0)|-|g'(0)|)C(\rho r)=r|z_2-z_1|m(0)C(\rho r)
$$
for all $z_1,z_2\in\gamma_\rho$. On the other hand for all $z_1,z_2\in\ID$,
\beqq
|\varphi(z_2)-\varphi(z_1)|&=&\left|\int_{z_1}^{z_2}p'(z)\,dz+\overline{\int_{z_1}^{z_2}q'(z)\,dz}\right|\\
&\le & \int_{z_1}^{z_2}(|p'(z)| +|q'(z)| )\,|dz| \\
&\le & |z_2-z_1|\, \sup_{z\in\ID}\big (|p'(z)|+ |q'(z)|\big)
= A|z_2-z_1|.
\eeqq
Taking into account of the above estimations for $z_1,z_2\in\gamma_\rho$ and $z_1\ne z_2$, we obtain that
\beqq
|F(z_2)-F(z_1)|&\ge &|f(rz_2)-f(rz_1)|-\varepsilon |\varphi(z_2)-\varphi(z_1)|\\
&\ge & |z_2-z_1|\big (rm(0)C(\rho r)-\varepsilon A \big )\\
&>&|z_2-z_1|\big (rm(0)C(r)-\varepsilon A \big ),
\eeqq
where the last inequality is a consequence of the monotonicity of $C(r)$ (see Lemma \ref{GPS2-th2n1}).

Therefore if $\varepsilon<m(0)C(r) (r/A)$, then $|F(z_2)-F(z_1)|>0$ for all $z_1,z_2\in\gamma_\rho$,  $z_1\ne z_2$, and
thus, for every $\rho\in(0,1)$, $F$ maps $\gamma_\rho$ univalently onto a simple closed curve. 
Applying the argument principle (see, for example \cite{Duren:Harmonic}),
we finally conclude that the function $F$ is univalent in $\ID$ and we complete the proof.
\epf

\section{Examples}\label{GPS2-sec4}

To illustrate the validity of Theorems  \ref{GPS2-th1} and \ref{GPS2-th3} consider following examples.

\beg
Let $h\in \mathcal{A}$ and such that ${\rm Re}\,(e^{i\alpha} h'(z))>0$ in $\ID$ for some $\alpha\in\Bbb R$.
Then $h$ is univalent in $\ID$ by the well-known Noshiro-Warschawski condition (cf. \cite[Chapter 2, Theorem 2.16]{Duren:Analytic})
and hence, the harmonic function
$$f_k(z)=h(z)+k\overline{h(z)}
$$
is univalent for any $k\in[0,1)$ as composition of univalent mappings. Note that Theorem \Ref{uni-theo1} is not applicable in this case. Indeed, if, for example, 
$h_0(z)=z+z^2/2$ and $\alpha=0$, then condition of Theorem \Ref{uni-theo1} takes the form
${\rm Re}\big ( e^{i\gamma} h_0'(z)\big)>k|h_0'(z)|$, that is,
$$ {\rm Re}\left ( e^{i\gamma} \frac{1+z}{|1+z|}\right)>k, ~\mbox{ i.e. }~ 
{\rm Re}\, e^{i(\gamma+\theta)} >k,
$$
where $\theta={\rm arg}(1+z)$ takes all values from $(-\pi/2,\pi/2)$. Obviously this condition can be fulfilled for all $z\in\ID$ and some fixed $\gamma$ only for $k=0$.

On the other hand the above function $h$ and the  condition \eqref{GPS2-eq2} of Corollary \ref{GPS2-cor1} with function
$$\phi (w)=-\frac{1}{k}w+\overline{w} 
$$ 
leads to the inequality
$${\rm Re}\big (e^{i\gamma}\big (k-\frac{1}{k}\big )h'(z)\big ) =\left (\frac{1-k^2}{k}\right ){\rm Re}\,(e^{i\alpha}h'(z))>0
$$
which is true for all $z\in\ID$ and  $\gamma=\alpha+\pi$. 
\eeg

The next less trivial example illustrates not only the limitations of the applicability of Theorem \Ref{uni-theo1} but also the utility of Theorem \ref{GPS2-th3}.

\beg
Consider  the function $h_1$ defined in $\ID$ by
$$
h_1(z)=\left (\frac{1+g(z)}{1-g(z)}\right )^2,
$$
where
$$g(z)=\sqrt{1-\frac{2}{1+\sqrt{3-\frac{8z}{(1+z)^2}}}}.
$$ 
Here branches of all square roots are principal.
It is a simple exercise to see that $h_1$ maps $\ID$ conformally and univalently onto $(\IC\setminus(-\infty,-1])\setminus\overline{\ID}$ (see Figure \ref{GPS2-fig}). Note that the function $h_1$ is not close-to-convex in $\ID$.
\begin{figure}
\begin{center}
\includegraphics[scale=1]{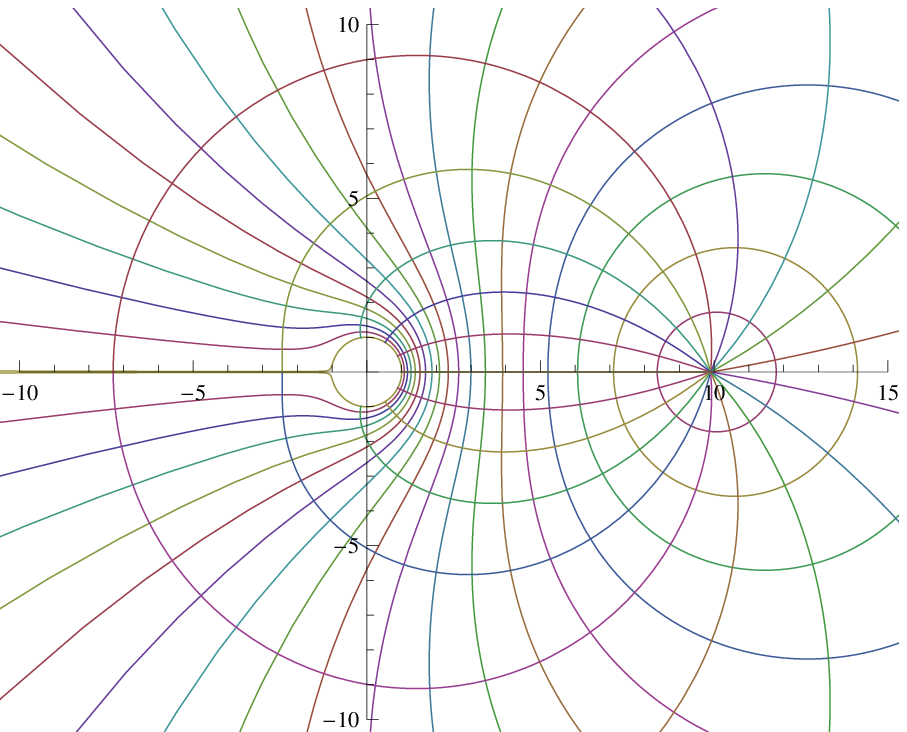}
\end{center}
\caption{Image  of $\ID$ under $h_1(z)$\label{GPS2-fig}}
\end{figure}

Now let $r\in(0,1)$ be fixed and, for $r$ close to $1$, consider the function $h_r(z)$ defined by
$$h_r(z)=h_1(rz).
$$
Let $\epsilon >0$ be a sufficient small number, and define
$$f_\epsilon (z)=h_r(z)+\epsilon\left(h_r(z)+ \overline{z}\right).
$$
In order to prove the univalence of $f_\epsilon$ for all sufficiently small $\epsilon$, we shall
apply Theorem \ref{GPS2-th3} to the function
$$F_\varepsilon(z)=h_r(z)+\varepsilon \overline{z}.
$$
It follows from Theorem \ref{GPS2-th3} that $F_\varepsilon$ is univalent in $\ID$ if
$$|\varepsilon|< r\min\{m(r),m(0)C(r)\}=:\varepsilon_0(r),
$$
where $m(r)=\min\{|h_1'(z)|:\,|z|\le r\}$, $C(r)$ is defined in Lemma \ref{GPS2-th2n1} and $\alpha$ in definition of $C(r)$ is equal to $2$, 
in view of Remark \ref{rem1}. Hence, $f_\epsilon (z)$ is also univalent for each $\epsilon$ with $|\epsilon|<\varepsilon_0(r)/(1-\varepsilon_0(r))$. 

Note that $\varepsilon_0(r)$ approaches $0$ as $r\to 1^-$.
Therefore, functions $f_\epsilon$ tend to $h_1$ uniformly on compact subsets of $\ID$ as $r\to 1^-$.
Hence the domain $f_\epsilon(\ID)$ tends to $h_1(\ID)$  in the sense of convergence to the kernel (cf. \cite[Chapter 2, \S 5]{Golu66})
and the domain $f_\epsilon(\ID)$ is not close-to-convex for $r$ sufficiently close to $1$.
Therefore, Theorem \Ref{uni-theo1} is not applicable in this case while Theorem \ref{GPS2-th3} allows us to state the univalence of $f_\epsilon$.

The univalence of function $f_\epsilon$ in this example can be also proved by Corollary \ref{GPS2-cor1}  with function
$\phi(w)=h_1^{-1}(w)$.
\eeg

\subsection*{Acknowledgements}
The research was supported by the project RUS/RFBR/P-163 under Department of Science \& Technology (India).
The second author is currently on leave from IIT Madras, India. The third author is also supported by 
Russian Foundation for Basic Research (project 14-01-00510) and the Strategic Development Program of Petrozavodsk State University.

\end{document}